\documentclass[11pt]{amsart}
\usepackage{preamble}
\begin{document}

\title[Adjoint path-kernel method for backpropagation and data assimilation]{
Adjoint path-kernel method for backpropagation and data assimilation in unstable diffusions
}

\begin{abstract}
We derive the adjoint path-kernel method for computing parameter-gradients (linear responses) of SDEs.
Its cost is almost independent of the number of parameters, and it works for non-hyperbolic systems with parameter-controlled multiplicative noise.
With this new formula, we extend the conventional backpropagation method to settings with gradient explosion, and demonstrate it on the 40-dimensional Lorenz 96 system.

Moreover, we consider a difficult version of the 4D-Var data assimilation problem where
(1) the deterministic part of the model is chaotic, 
(2) the loss is a single long-time functional accounting for discrepancies in both the observations and the dynamics,
(3) some parameters in the dynamics are unknown, and 
(4) some coordinates of the states cannot be observed, and cannot be reasonably inferred from other coordinates within a short time.
We model the correction term at each time-step separately as a parameterized function of the random state.
With our new tool, we can run stochastic gradient descent to find the path and parameters that best match the low-dimensional observation data.
We demonstrate this on the 10D Lorenz-96 system with 8D observations.

\smallskip
\noindent \textbf{MSC2020 subject classification numbers.}
60H07, % Stochastic calculus of variations and the Malliavin calculus
65D25, % Numerical differentiation
37M25, % Computational methods for ergodic theory (approximation of invariant measures, computation of Lyapunov exponents, entropy, etc.)
62D10. % Missing data
%60J60, % Diffusion processes
%65C30, % Numerical solutions to stochastic differential and integral equations 
%\kwd[; secondary ]{65C05} % Monte Carlo methods
%\kwd{60G10} % Stationary stochastic processes

\smallskip
\noindent \textbf{Keywords.}
Linear response,
Backpropagation,
Likelihood ratio,
%Diffusion process,
Data assimilation,
Gradient descent.
\end{abstract}

\maketitle

\section{Introduction}
\label{s:intro}

\subsection{Main results}

This paper rigorously derives the adjoint path-kernel formula for the parameter-gradient of discrete-time random dynamical systems in \Cref{t:adj_discrete}.
Its cost is independent of the number of parameters, so it is suitable for cases with many parameters.
Then we formally pass to the continuous-time limit in \Cref{t:adjSDE}.
We also formally derive the adjoint path-kernel formula for the parameter-gradient of stationary measures in \Cref{t:adjInf}.

\begin{restatable}[adjoint continuous-time path-kernel]{formula}{goldbach}
\label{t:adjSDE}
For any $x_0$, $v_0$, and adapted scalar process $\alpha_t$,
consider the Ito SDE,
\[
dx^\gamma_t
= F^\gamma(x^\gamma_t) dt 
+ \sigma^\gamma(x^\gamma_t)dB,
\quad
x_0^\gamma =  x_0 + \gamma v_0.
\]
Let $\nu_t$ be the backward covector process of the damped adjoint equation,
\begin{equation*}\begin{split} 
  - d\nu 
  = 
  - \alpha \nu dt 
  + \nabla F_k^T\nu dt +  \nabla \sigma (x) \nu^T dB
  + (\Phi(x_T)-\Phi^{avg}_T) \alpha_t dB / \sigma(x)
\end{split}\end{equation*}
with terminal condition $\nu_{T} = \nabla \Phi(x_T)$.
Then the linear response has the expression
\[
\delta \E{\Phi(x^\gamma_T)} 
= \E{\nu_0 \cdot v_0
+ 
\int_{t=0}^{T} \nu_t \cdot \left(\delta F(x) dt + \delta \sigma(x) dB\right)  }.
\]
\end{restatable}

Here $\delta(\cdot):= \partial (\cdot)/\partial \gamma |_{\gamma=0}$,
$\Phi^{avg}_T:=\E{\Phi(x^{\gamma=0}_T)}$,
$B$ is the Brownian motion, the SDE is Ito, and the integrations of backward processes are the limits of \Cref{e:liu,e:zhao}.
Similarly to the tangent version in \cite{dud}, the adjoint version here has the following advantages: (1) $\sigma$ can depend on $x$ and $\gamma$; (2) $\nu$ does not grow exponentially over time since $\alpha$ damps Lyapunov exponents; (3) it does not assume hyperbolicity.

The main advantage compared to \cite{dud} is that the cost is almost independent of the number of parameters.
This is because when we have multiple parameters $\gamma$, such as in the case of neural networks, the derivative with each parameter uses the same main term $\nu$, so it only needs to be computed once.
These new formulas enable Monte-Carlo-type computation of the parameter-gradient of unstable diffusions in high dimensions.
For example, we use it to compute the parameter-gradient of the Lorenz-96 system with multiplicative noise, which cannot be solved by previous algorithms.

With the new adjoint path-kernel tool, we develop a new framework for a particularly difficult variant of 4D-Var data assimilation problems. 
Our setting combines several sources of hardness: 
(1) the deterministic part of the dynamics is unstable/chaotic, 
(2) the loss is a single function of the entire space-time path over a long window, 
(3) some parameters in the dynamics are unknown, and 
(4) only a low-dimensional projection of the state is observed, with unobserved coordinates that cannot be reliably inferred from a short time segment.
On top of this, we model the correction term at each time step as a parameterized (but non-random) function of the random state, so the overall control consists of many parameters describing these corrections along the trajectory; our tangent path-kernel method in \cite{dud} becomes too expensive for this case.

Our adjoint path-kernel method provides a stable way to differentiate the long-time loss with respect to the initial condition of the latent path, all parameters in the dynamics (including the diffusion), and corrections on all time steps.
Hence, we can run stochastic gradient descent even in this high-dimensional, strongly chaotic, partially observed regime. 
As a proof of concept, we demonstrate that this approach successfully recovers both the path and dynamics in the three-dimensional Lorenz-63 system from only two observed coordinates, showing that the method can handle long-horizon, pathwise learning problems that are difficult for standard 4D-Var or filtering formulations.

\subsection{Literature review}

\subsubsection{Linear response methods}

The averaged statistic of a random dynamical system is of central interest in applied sciences.
It is a fundamental tool for many applications in statistics and computing.
There are three basic methods for expressing and computing derivatives of the marginal or stationary distributions of random systems: the path-perturbation method (shortened as the path method), the divergence method, and the kernel-differentiation method (shortened as the kernel method).
These methods can be used for derivatives with respect to the terminal conditions, initial conditions, and parameters of dynamics (known as the linear response).
The relation and difference among the three basic methods can be illustrated in a one-step system, which is explained in \cite{Ni_kd}.

The path-perturbation method is also known as the ensemble method or the stochastic gradient method \cite{eyink2004ruelle,lucarini_linear_response_climate}.
It also includes the backpropagation method, which is the basic algorithm for machine learning.
It is good at stable systems and derivatives on initial conditions.
However, it is expensive for chaotic or unstable systems; the workaround is to artificially reduce the size of the path-perturbation, such as shadowing or clipping methods \cite{Ni_nilsas,Ni_NILSS_JCP,clip_gradients2,FP25}, but they all introduce systematic errors.
Also, the path-perturbation only computes the `weak' linear response, which means that we need to first integrate the measure with an observable function and then take derivative; it does not directly express the derivative of the measure or the transfer operator.

The divergence method directly expresses the derivative of the measure or the transfer operator; this is the transfer operator method in deterministic settings.
It is good at unstable systems and derivatives of marginal densities.
Traditionally, for systems with contracting directions, the recursive divergence formula grows exponentially fast, so the cost of Monte-Carlo-type algorithm is high for long times.
The workaround is to use a finite-element-type algorithm, which has deterministic error rather than random sampling error, but is expensive in high dimensions \cite{Froyland2007,Galatolo2014,Wormell2019a,Zhang2020}.

The kernel-differentiation method also directly expresses the derivative of the measure or the transfer operator, but only for random systems, since the derivative hits the probability kernel.
In SDEs, this formula is a direct result of the Cameron-Martin-Girsanov theorem \cite{CM44,MalliavinBook,score14}; it is also called the likelihood ratio method or the Monte-Carlo gradient method \cite{Rubinstein1989,Reiman1989,Glynn1990}.
This formula also works for stationary measures, and proofs were given in \cite{HaMa10,GG19}.
It is good at taking derivative for random systems whose deterministic part has poor dynamical properties, such as non-hyperbolicity.
However, it cannot handle multiplicative noise or perturbation on the diffusion coefficients.
It is also expensive when the noise is small.

Mixing two basic methods can overcome some major shortcomings.
For hyperbolic systems, the fast response formula uses the path-perturbation method in the stable, and the divergence method in the unstable \cite{Ni_asl,TrsfOprt,fr,vdivF,GN25}. 
It is good at high dimensions and no-noise system \cite{Ni_nilsas,far}.
However, it does not work when the hyperbolicity is poor \cite{Baladi2007,wormell22}.

We can also mix the path-perturbation with the kernel methods.
The Bismut-Elworthy-Li formula \cite{Bismut84,EL94,HM06,PW19bismut} computes the derivative with respect to the initial conditions, but it does not handle $dB$-type perturbations caused by changing the diffusion.
The path-kernel method in \cite{dud} gives the linear response of the diffusion coefficients.
It is good at systems with not too small noise and not too much unstableness, it does not require hyperbolicity, and it can handle perturbation on initial conditions.
However, it can be expensive when the noise is small and unstableness is large.
Moreover, its cost increases linearly with the number of parameters; hence, it is not suitable for cases like the variational data assimilation problem in this paper, which has many parameters to be optimized.
The second issue can be solved by the adjoint formulation in this paper.

The paper \cite{divker} should be the first example mixing the divergence and kernel-differentiation methods.
Such a mixture is good at systems with not too much contraction and not too small noise; it allows multiplicative noise; it does not require hyperbolicity.
Moreover, it naturally handles the score, which is the derivative of marginal densities.
In \cite{DKlinR}, we further give a pathwise expression of the linear response, and use it to give a new framework for generative models.
The problem of generative model is related to but different from data assimilation: generative model infers the dynamics from many samples without an ordering in time, whereas data assimilation infers the dynamics from a sequence of data.
As to the divergence-kernel method, it does not handle perturbations on singular initial conditions well, so it does not work on the version of the data assimilation problem we consider.

There are other results that do not fall into our logic.
Some involve working in some abstract spaces beyond the basic path spaces, so they involve more complicated terms \cite{MG25score}.
Some have singularities in the dynamics, so they involve extra terms at the singularity \cite{wormell_pieceMap}.
Nevertheless, they have the same problem if they involve terms from the above methods.

We also proposed a triad program in \cite{Ni_kd}, which requires advancing and mixing all three methods.
That might be the best solution for computing derivatives of random systems or approximate derivatives of deterministic systems.

The first way to view the significance of the adjoint path-kernel formula in this paper is that we derive the adjoint version of the path-kernel method.
Practically speaking, here, adjoint means that the main term is shared for multiple $\gamma$.
So, the cost of computing the derivative with respect to many parameters is low.
The backpropagation method in machine learning is an adjoint method.
The second way to view this new formula is that we extend the backpropagation method to work in cases with gradient explosion.
The third way to view it is that we explicitly give the terms missing from clipping methods.
Fourth, the adjoint path-kernel formula enables a new 4D-Var framework for chaotic systems with partial observation and unknown parameters, which we demonstrate on Lorenz 63 and Lorenz 96.

\subsubsection{Data assimilation}

Data assimilation (DA) provides a mathematical and algorithmic framework for combining imperfect dynamical models with noisy, partial observations. 
In geoscience and related areas, it is used to reconstruct the evolving state of the atmosphere and ocean, estimate poorly known parameters, and build physically consistent reanalyses for climate studies and prediction \cite{DA_Kalnay,DA_AST}.

Within this general framework, one can distinguish sequential filtering methods from variational or 4D-Var methods. 
Filtering methods (Kalman-type filters, ensemble Kalman filters, particle filters, etc.) propagate an estimate of the state forward in time and then update it when new data arrive \cite{filter_ABN}. 
Conceptually, they solve many short-horizon problems, each associated with a single observation time, by combining the model forecast and the new observation. 
Hence, the effect of each observation is mostly local in time, and model error is often absorbed into process noise or inflation.
Moreover, a sequential filter does not typically provide an accurate initial condition at the start of the window. 
The early-time estimates are driven by a broad or misspecified prior and only become accurate after enough observations have been assimilated. 
On the other hand, this stepwise structure makes filtering applicable in cases with partial observation and unknown parameters in dynamics \cite{SG25_filter}.

By contrast, four-dimensional variational assimilation (4D-Var) explicitly optimizes over the entire trajectory (and possibly parameters), yielding a good initial condition and a dynamically consistent path that fits all observations simultaneously.
In this sense, the 4D-Var problem is `harder' and more ambitious than filtering: rather than solving many loosely coupled short-horizon problems, it demands one space-time path that simultaneously explains all observations and is dynamically consistent. 
At the same time, the long-window coupling makes the optimization problem more challenging, especially in chaotic, partially observed systems.

The main ingredient in the 4D-Var framework is to compute the gradient of the long-window cost function with respect to initial conditions by integrating adjoint equations \cite{4dvar_DT,4dvar_TC}. 
Despite its conceptual appeal, long-window 4D-Var is notoriously difficult to use in the regimes we care about. 
A first obstacle is partial observation and unknown parameters: when only a low-dimensional projection of the state is observed and key parameters in the dynamics are not known a priori, any natural initial guess for the optimization is typically very far from the truth or the basin of the global minimum.
Some recent 4D-Var results, in the current form of their analysis, avoid dealing with this obstacle by requiring observation of the full state  \cite{Maulik24,ShiLi25}.
This compounds with the second obstacle, long assimilation windows on a chaotic system, where the 4D-Var loss has many local minima \cite{Cox2015,Brocker2017}, and the backpropagation used to compute gradients suffers from exponential growth, leading to an exploding demand for samples.

First, we temporarily add noise to the dynamics to smooth out spurious local minima in the loss. 
Second, we use the adjoint path-kernel method, whose expression does not grow with time, to compute gradients of the long-time loss with respect to the initial condition, the parameter, and the correction term at each time step.
This allows us to run stochastic gradient descent starting from very rough initial guesses and still recover trajectories and parameters that both fit the observations and remain close to dynamically consistent paths, in settings where traditional 4D-Var formulations struggle.

Besides solving the 4D-Var problem, our method also addresses a trajectory reproducibility issue.
We optimize a collection of time-dependent correction terms, one for each time step; after optimization, the optimized correction terms temper chaos by attracting the trajectory toward the true orbit.
As a result, when we rerun the optimized dynamics, the simulated trajectory remains close to the true trajectory over the assimilation window.
In fact, in our numerical experiments, after removing the artificial noise and rerunning the corrected deterministic system, the optimized solution matches the observations even better and exhibits smaller dynamical discrepancy.

\subsection{Structure of the paper}

\Cref{s:notations} defines some basic notation and reviews the tangent version of the path-kernel method.
\Cref{s:derive} derives the adjoint results for discrete-time systems, then formally passes to the continuous- and infinite-time limit.
\Cref{s:numeric} considers numerical realizations, where we compute the linear response of the stationary measure of the 40-dimensional Lorenz 96 model with multiplicative noise.
This example cannot be solved by previous adjoint methods.

\Cref{s:assi} solves the 4D-Var data assimilation problem for long-time chaotic systems with unknown parameters and partial observation.
This problem uses gradient-based optimization, and the adjoint path-kernel method is the main tool; it cannot be solved by other gradient methods including the author's tangent path-kernel method and the divergence-kernel method.
\Cref{s:assi_desp} defines the dynamics, the loss function, and the step-wise correction terms.
\Cref{s:assi_adj} derives the adjoint path-kernel formula adapted to this specific problem.
\Cref{s:assi_opti} explains the optimization strategy and how to set learning rate.
\Cref{s:assi_3D} applies our solution on the 3D Lorenz 63 system with 2D observation and 1 unknown parameter.
\Cref{s:assi_10D} applies our solution on the 10D Lorenz 96 system with 8D observation and 1 unknown parameter.

\section{Notations and Preparations}
\label{s:notations}

We define some geometric notations.
Denote both vectors and covectors by column vectors in $\R^M$; the product between a covector $\nu$ and a vector $v$ is denoted by $\cdot$, 
that is,
\begin{equation*}\begin{split}
  \nu\cdot v 
  := v\cdot \nu
  :=\nu^T v
  :=v^T \nu.
\end{split}\end{equation*}
Here $v^T$ is the transpose of matrices or (co)vectors.
Note that the Brownian increment $\DB$ may be either a vector or a covector.
Denote
\begin{eqnarray*}
  \nabla (\cdot):=\pp{(\cdot)}x,\quad
  \nabla_v (\cdot):=
  \nabla (\cdot) v :=
  \pp{(\cdot)}x v,
\end{eqnarray*}
Here $\nabla_YX$ denotes the (Riemann) derivative of the tensor field $X$ along the direction of $Y$.
It is convenient to think that $\nabla$ always adds a covariant component to the tensor.
For a map $g$, let $\nabla g$ be the Jacobian matrix, or the pushforward operator on vectors.

In \cite{dud}, we rigorously derived the path-kernel formula for the linear response of discrete-time random dynamical systems.
Let $\gamma$ be the parameter that controls the dynamics, the initial condition, and hence the distribution of the process $\{x^\gamma_n\}_{n\ge0}$; by default $\gamma=0$, so $x:=x^{\gamma=0}$.
We denote the perturbation $\delta(\cdot):=\partial (\cdot)/ \partial\gamma|_{\gamma=0}$.
Let $\Phi$ be a fixed $C^2$ observable function.
Assume that the drift $f$ and diffusion $\sigma$ are $C^1$ functions and $C^1$-depend on $\gamma$.
Note that the tangent equation of $v$ depends on the path $x$ and the corresponding $\{ b_n \}_{ n\ge0 }$ that drives $x$.
In the following, let $\cF_n$ be the sigma-algebra generated by $x_0, b_0, \ldots, b_{n-1}$.

\begin{theorem}[tangent discrete-time path-kernel]
\label{t:tan_discrete}
Fix any $x_0$, $v_0$, and any $\alpha_n$ (called a `schedule') a scalar process adapted to $\cF_n$ and independent of $\gamma$.
Consider the random dynamical system,
\begin{equation*} 
  x^\gamma_{n+1}
  =f^\gamma(x^\gamma_n) + \sigma^\gamma(x^\gamma_n) b_n,
  \quad
  x_0^\gamma = x_0 + \gamma v_0,
  \quad
  b_n \overset{i.i.d.}{\sim} \cN(0,I). 
\end{equation*}
Note that $f^\gamma(\cdot)$ and $\sigma^\gamma(\cdot)$ depend on the parameter $\gamma$.
Let $v_n$ be the solution of the following tangent equation starting from $v_0$
\begin{equation*}
  v_{n+1} 
  = - \alpha_n v_n
  + \nabla_{v_n}f(x_n) + \delta f^\gamma(x_n) 
  + (\nabla_{v_n}\sigma(x_n) + \delta \sigma^\gamma (x_n)) b_n.
\end{equation*}
Denote $\Phi^{avg}_N:=\E{\Phi(x_N)}$, the linear response has the expression
\[
\delta \E{\Phi(x^\gamma_N)}
= 
\E{ \nabla\Phi(x_N) \cdot v_N
+(\Phi(x_N)-\Phi^{avg}_N)\sum_{n=0}^{N-1} \frac{b_n }{\sigma(x_n)}
\cdot \alpha_n v_n.
}
\]
\end{theorem}

The above result has \textit{no} approximation.
Then we formally passed to the continuous-time limit.
We assume that all integrations, averages, and change of limits are legitimate.
A fully rigorous proof would require additional regularity and mixing assumptions which we do not pursue here.
Here, $B$ denotes the Brownian motion.
In the formula below, typically $\alpha_t \ge 0$, so the term $\alpha_tv_t dt$ damps the unstable growth of the path-perturbation $v_t$; it is the portion of the path-perturbation shifted to the probability kernel.

\begin{formula}[tangent continuous-time path-kernel]
\label{t:tan_sde}
Fix any $x_0$, $v_0$, and adapted scalar process $\alpha_t$,
consider the Ito SDE,
\[
dx^\gamma_t
= F^\gamma(x^\gamma_t) dt 
+ \sigma^\gamma(x^\gamma_t)dB,
\quad
x_0^\gamma = x_0^\gamma :=  x_0 + \gamma v_0.
\]
Let $v_t$ be the solution of the damped tangent equation starting from $v_0$,
\[
dv 
= - \alpha_t v dt 
+ \left(\nabla_v F(x) + \delta F^\gamma (x) \right) dt
+ \left(d\sigma(x)v + \delta \sigma^\gamma (x)\right)dB.
\]
Then the linear response has the expression
\[
\delta \E{\Phi(x^\gamma_T)}
= 
\E{ d\Phi(x_T) v_T
+(\Phi(x_T)-\Phi^{avg}_T) \int_{t=0}^{T} \frac{\alpha_t v_t }{\sigma(x_t)}\cdot dB_t}.
\]
\end{formula}

Then we present the linear response formula, on a single orbit of infinite time, for the stationary measure.
When we run the SDE for an infinitely long time, if the probability does not leak to infinitely far away, then the distribution of $x_t$ typically converges weakly to the stationary measure $\mu$.
By the ergodic theorem, for any smooth observable function $\Phi$ and any initial condition $x_0$,
\[
\E[\mu^\gamma]{\Phi(x)}:= \int \Phi(x)d\mu^\gamma(x)
:=\lim_{T\rightarrow\infty} \E{\Phi(x^\gamma_T)}
\,\overset{\mathrm{a.s.}}{=}\,
\lim_{T\rightarrow\infty} 
\frac 1T \int_{t=0}^T \Phi(x^\gamma_t) dt.
\]

The following corollary formula was derived by letting $T\rightarrow\infty$, then applying the decay of correlations and the exponential decay of the propagation of the tempered tangent equation.
Let $\Phi^{avg}:=\E[\mu]{\Phi(x)}$.
Let $W$ indicate the decorrelation and $T$ the orbit length, typically $W\ll T$ in numerics,
\begin{formula}[tangent ergodic path-kernel]
\label{t:tan_ergodic}
\[
\delta \E[\mu^\gamma]{\Phi(x)}
\,\overset{\mathrm{a.s.}}{=}\,
\lim_{W\rightarrow\infty} \lim_{T\rightarrow\infty} 
\frac 1T \int_{t=0}^T \left[
d\Phi(x_t) v_t
+ \left( \Phi(x_{t+W}) - \Phi^{avg} \right)
\int_{\tau=0}^W  \frac{\alpha_{t+\tau} v_{t+\tau} }{\sigma(x_{t+\tau})} \cdot dB_{t+\tau} 
\right]dt.
\]
\end{formula}

\section{Deriving the adjoint} \label{s:derive}

\subsection{Discrete-time adjoint} \label{s:adj_discrete}

\begin{theorem}[adjoint discrete-time path-kernel]
\label{t:adj_discrete}
For any $x_0, v_0\in\R^M$, and any $\alpha_n$ (called a `schedule') a scalar process adapted to $\cF_n$ and independent of $\gamma$.
Consider the random dynamical system,
\begin{equation*}
  x^\gamma_{n+1}
  =f^\gamma(x^\gamma_n) + \sigma^\gamma(x^\gamma_n) b_n,
  \quad
  x_0^\gamma = x_0 + \gamma v_0,
  \quad
  b_n \overset{i.i.d.}{\sim} \cN(0,I). 
\end{equation*}
Define the backward covector process $\nu$ (it becomes deterministic once a path $\{x_n\}$ is fixed)
\begin{equation*}\begin{split}
  \nu_{N} = \nabla \Phi(x_N),
  \quad \textnormal{} \quad 
  \nu_{k} = 
  - \alpha_k \nu_{k+1} + 
  (\nabla f_k^T +  \nabla \sigma_k b_k^T) 
  \nu_{k+1} + (\Phi(x_N)-\Phi^{avg}_N) \alpha_k b_k / \sigma_k.
\end{split}\end{equation*}
The subscript $k$ means to evaluate at $x_k$ when needed.
Then, the linear response can be expressed by
\begin{equation*}\begin{split}
\delta \E{\Phi(x^\gamma_N)} 
= \E{\nu_0 \cdot v_0
+ 
\sum_{k=0}^{N-1} \nu_{k+1} \cdot \left(\delta f_k + \delta \sigma_k b_k\right)  }.
\end{split}\end{equation*}
\end{theorem}

\begin{proof}
We can obtain a pathwise tangent-adjoint equivalence.
On each path, $\{ x_n \}_{ n=0 }^N$ and $\{ b_n \}_{ n=0 }^N$ are known, so the tangent equation of $v_n$ in \Cref{t:tan_discrete} becomes deterministic, which we shorten as
\begin{equation*}\begin{split}
  v_{n+1} = M_n v_n + p_{n+1},
  \quad \textnormal{where} \quad 
  M_n := - \alpha_n I  + \nabla f_n + b_n \nabla \sigma_n^T,
  \quad
  p_{n+1} := \delta f_n + \delta \sigma_n b_n,
\end{split}\end{equation*}
Here we used $
\nabla_{v} \sigma b 
= b (\nabla \sigma^T v) 
= (b \nabla \sigma^T ) v.
$
Here $p_n$ is a vector at $x_n$.
Note that $\delta f_n$ is a vector at $x_{n+1}$.
This equation is affine in $v$, so we can write out the expansion of $v_n$ for $n\ge 1$,
\begin{equation*} \begin{split}
  v_n = M_{n-1} \cdots M_{0} v_0
  + \sum_{k=1}^{n} M_{n-1} \cdots M_{k} p_k,
\end{split} \end{equation*}
where the sum is zero for $n=0$, so $v_0=v_0$.

By \Cref{t:tan_discrete}, the linear response has the following expression
\begin{equation*}\begin{split}
\delta \E{\Phi(x^\gamma_N)} = \E{ \sum_{n=0}^{N} \omega_n \cdot v_n },
\quad \textnormal{where} \quad 
\\
\omega_N := \nabla \Phi(x_N),
\quad \textnormal{} \quad 
\omega_n := (\Phi(x_N)-\Phi^{avg}_N) \alpha_n b_n / \sigma_n.
\end{split}\end{equation*}
Substituting the expansion of $v_n$ and transposing matrices, we have
\begin{equation*}\begin{split}
\delta \E{\Phi(x^\gamma_N)} 
= \E{ \sum_{n=0}^{N} \omega_n \cdot \left(
M_{n-1} \cdots M_{0} v_0
+ \sum_{k=1}^{n} M_{n-1} \cdots M_{k} p_k \right) }
\\
= \E{ \sum_{n=0}^{N} M_{0}^T \cdots M_{n-1}^T \omega_n \cdot v_0
+ 
\sum_{n=1}^{N} \sum_{k=1}^{n} M^T_{k} \cdots M^T_{n-1} \omega_n \cdot p_k },
\end{split}\end{equation*}
Interchange the order of summation, we get
\begin{equation*}\begin{split}
\delta \E{\Phi(x^\gamma_N)}
= \E{ \sum_{n=0}^{N} M_{0}^T \cdots M_{n-1}^T \omega_n \cdot v_0
+ \sum_{k=1}^{N} \sum_{n=k}^{N} M^T_{k} \cdots M^T_{n-1} \omega_n \cdot p_k },
\end{split}\end{equation*}

Define the backward covector process $\nu$ (on this path it is also deterministic) by
\begin{equation*}\begin{split}
  \nu_{N} = \omega_N,
  \quad \textnormal{} \quad 
  \nu_{k} = M_k^T \nu_{k+1} + \omega_k.
\end{split}\end{equation*}
So the $\nu_n$ has the expansion
\begin{equation*}\begin{split}
  \nu_k = \sum_{n=k}^{N}  M^T_k \cdots M^T_{n-1} \omega_n.
\end{split}\end{equation*}
Hence, the linear response can be expressed by
\begin{equation*}\begin{split}
\delta \E{\Phi(x^\gamma_N)} 
= \E{\nu_0 \cdot v_0
+ 
\sum_{k=1}^{N} \nu_k \cdot p_k },
\end{split}\end{equation*}
The theorem is proved by substituting the definitions of $p, M$, and $\omega$.
\end{proof}

\subsection{Continuous-time adjoint} \label{s:cts}

We \textit{formally} pass the discrete-time results to the continuous-time limit SDE to derive \Cref{t:adjSDE}.
Let $\cF_t$ be the $\sigma$-algebra generated by $\{B\tau\}_{\tau\le t}$ and $x_0$.
We take $\alpha_t$ to be a scalar process adapted to $\cF_t$.
We also assume that $\alpha_t$ is integrable with respect to $dB_t$.

%\goldbach*

\begin{derivation}[Derivation of \Cref{t:adjSDE}]
Our derivation is performed on the time span divided into small segments of length $\Dt$.
Let $N$ be the total number of segments, so $N\Dt = T$.
Denote
\[ 
  \DB_n:= B_{n+1} - B_n.
\]
Denote $\alpha_n = \alpha_{n\Dt}$.
The discretized SDE is 
\begin{equation*} 
  x_{n+1} - x_{n} = F(x_n) \Dt + \sigma(x_n)\DB_n.
\end{equation*}
Comparing with \Cref{t:adj_discrete} (whose  $\alpha$ and $\sigma$ are denoted by $\alpha'$ and $\sigma'$ here), we have
\begin{equation*}\begin{split}
  f(x) := x+ F(x)\Dt, \quad
  \sigma'(x) := \sigma(x) \sqrt{\Dt}, \quad 
  b_n := \DB_n /\sqrt{\Dt}, \quad
  \alpha'_n := \alpha_n \Dt.
\end{split}\end{equation*}

So, the terminal condition of $\nu$ becomes $ \nu_{T} = \nabla \Phi(x_T)$, and its backward equation becomes
\begin{equation}\begin{split} \label{e:liu}
  \nu_{k} = 
  - \alpha'_k \nu_{k+1} + 
  (\nabla f_k^T +  \nabla \sigma'_k b_k^T) 
  \nu_{k+1} + (\Phi(x_N)-\Phi^{avg}_N) \alpha'_k b_k / \sigma'_k
  \\
  = \nu_{k+1}
  - \alpha_k \nu_{k+1} \Dt + 
  (\nabla F_k^T \Dt +  \nabla \sigma_k \DB_k^T) 
  \nu_{k+1} + (\Phi(x_N)-\Phi^{avg}_N) \alpha_k \DB_k / \sigma_k.
\end{split}\end{equation}
Then, the expression of the linear response becomes
\begin{equation}\begin{split} \label{e:zhao}
\delta \E{\Phi(x^\gamma_N)} 
= \E{\nu_0 \cdot v_0
+ 
\sum_{k=0}^{N-1} \nu_{k+1} \cdot \left(\delta F_k \Dt + \delta \sigma_k \DB_k\right)  }.
\end{split}\end{equation}
Then we formally pass to the limit $\Dt\rightarrow0$.
\end{derivation}

\subsection{Infinite-time adjoint} \label{s:infinite}

Then we formally derive the adjoint linear response formula of stationary measures.

\begin{formula} [formal adjoint infinite-time path-kernel] \label{t:adjInf}
Assume there is only one stationary measure for the SDE
\[
dx^\gamma_t
= F^\gamma(x^\gamma_t) dt 
+ \sigma^\gamma(x^\gamma_t)dB.
\]
If we solve the backward adjoint equation with zero terminal condition $\nu_{T}=0$,
\begin{equation*}\begin{split} 
  - d \nu_t
  = 
  - \alpha_t \nu_t dt + \nabla F^T_t \nu_t dt +  \nabla \sigma_t \nu^T_t dB_t
  + \nabla \Phi_t dt + 
  \frac{\alpha_t} { \sigma_t } 
  \left(\int_{\tau=0}^W \left( \Phi_{t+\tau} - \Phi^{avg} \right) dt \right)
  dB_t.
\end{split}\end{equation*}
Then the linear response has the expression
\[
\delta \E[\mu^\gamma]{\Phi(x)}
\,\overset{\mathrm{a.s.}}{=}\,
\lim_{W\rightarrow\infty} 
\lim_{T\rightarrow\infty} 
\frac 1{T} \int_{t=0}^{T} 
\nu_t \cdot
\left[ \delta F^\gamma_t dt + \delta \sigma^\gamma_t dB_t \right].
\]
Here the integrations of the backward processes are the limits of \Cref{e:nu,e:infi_response}.
\end{formula}

\begin{derivation}
The time-discretized version of \Cref{t:tan_ergodic} is, for the SDE
\[
x^\gamma_{n+1}
= x^\gamma_n + F^\gamma(x^\gamma_{n}) \Dt 
+ \sigma^\gamma(x^\gamma_{n}) \DB_n,
\]
let $v_n$ be the solution of tangent equation
\begin{equation*}\begin{split}
v _{n+1}
= v_n - \alpha_n v_n \Dt 
+ \nabla F_n v_n \Dt + \nabla\sigma_n^T v_n \DB_n
+ p_{n+1}, \\
\quad \textnormal{where} \quad 
p_{n+1}:= \delta F^\gamma_n \Dt + \delta \sigma^\gamma_n \DB_n.
\end{split}\end{equation*}
The initial condition does not matter since $x, v$ converges to stationary measure, so we set $v_0=0$.
Then the linear response has the expression
\[
\delta \E[\mu^\gamma]{\Phi(x)}
\,\overset{\mathrm{a.s.}}{=}\,
\lim_{N_W\rightarrow\infty} \lim_{N\rightarrow\infty} 
\frac 1{N} \sum_{n=0}^{N-1} \left[
\nabla \Phi_n v_n
+ \left( \Phi_{n+N_W} - \Phi^{avg} \right)
\sum_{m=0}^{N_W-1}
\frac{\alpha_{n+m} v_{n+m} }{\sigma_{n+m}} \cdot \DB_{n+m} 
\right].
\]
Here $N=T/\Dt$, $N_W = W/\Dt$, where $W$ is the decorrelation length.
Collecting $v_n$ at the same time step, then divide and multiply by $\Dt$, we get
\begin{equation*}\begin{split}
\delta \E{\Phi(x^\gamma)}
\,\overset{\mathrm{a.s.}}{=}\,
\lim_{N_W\rightarrow\infty} \lim_{N\rightarrow\infty} 
\frac 1{N\Dt} \sum_{n=0}^{N-1} 
v_n \cdot\omega_n \Dt
\\
\quad \textnormal{where} \quad 
\omega_n := 
\nabla \Phi_n
+ 
\frac{\alpha_n}{\sigma_n} \DB_n 
\sum_{m=1}^{N_W}
\left( \Phi_{n+m} - \Phi^{avg} \right)
\end{split}\end{equation*}

By the same argument as in the proof of \Cref{t:adj_discrete},
if we solve the backward adjoint equation with zero terminal condition $\nu_{N}=0$,
\begin{equation}\begin{split} \label{e:nu}
  \nu_{k} 
  = \nu_{k+1}
  - \alpha_k \nu_{k+1} \Dt 
  + (\nabla F_k^T \Dt +  \nabla \sigma_k \DB_k^T) \nu_{k+1} 
  + \omega_k \Dt.
\end{split}\end{equation}
Then, on this path, we have the exact equivalence
\begin{equation*}\begin{split}
  \sum_{n=0}^{N-1} 
  v_n \cdot \omega_n \Dt
  =   
  \sum_{k=0}^{N-1} 
  p_{k+1} \cdot \nu_{k+1}.
\end{split}\end{equation*}
Hence, the linear response has the expression
\begin{equation}\label{e:infi_response}
  \delta \E[\mu^\gamma]{\Phi(x)}
  \,\overset{\mathrm{a.s.}}{=}\,
  \lim_{N\rightarrow\infty} 
  \frac 1{N\Dt} \sum_{k=0}^{N-1} \left[
  \left(\delta F^\gamma_k\Dt + \delta \sigma^\gamma_k \DB_k \right)  
  \cdot \nu_{k+1} 
  \right].
\end{equation}
Then we formally pass to $\Dt\rightarrow 0$.
\end{derivation}

\subsection{How to use}

We discuss how to use the adjoint path-kernel formulas.
The discussion of the tangent version in \cite{dud} also applies to the adjoint version in this paper.
Roughly speaking, we set $\alpha$ to be larger than the largest Lyapunov exponent.
If we care much about cost, we can further let $\alpha$ take different values based on $x$.
To ultimately reduce the cost, we should involve the divergence method, and a preliminary result is given in \cite{divker}.

For the linear response of stationary measures, when using \Cref{t:adjInf} in practice, to accelerate convergence, we should throw away some steps at the start and end of the path in $[0,T]$.
Because for $t\in[0, W]$, each $\omega_n$ multiplies with less than $N_W$ many $p_n$'s.
For $t\in[T-W, T]$, each $p_n$ multiplies with less than $N_W$ many $\omega_n$'s.
Our assumption of decorrelation basically requires that each $p_n$ multiplies with the next $N_W$ many $\omega_n$'s, and we can ignore the rest.
Hence, the contributions from these two time spans tend to have smaller absolute values than average.
We should first compute $v$ or $\nu$ on $[0,T]$, throw away the part in the time span $[0,W]$ and $[T-W,T]$, then compute the product and take the average.

Our formula involves a forward process of $x_n$'s and then a backpropagation process of $\nu_n$'s.
It seems that backpropagation requires us to record all $\DB_n$'s generated during the forward process.
We cannot use the conventional checkpoint trick for conventional adjoint methods in deterministic systems, which stores $x_n$ occasionally, then recover a small segment of the path when the backpropagation reaches this segment.
In random systems, we cannot calculate $x_{n+1}$ from only knowledge of $x_n$; we must also know $\DB_n$, which cannot be obtained unless we remember it during the forward run.
This extra memory cost might be regarded as unacceptable in some applications, such as fluid optimization; for these cases, we might need to compute parameter-gradient on low-fidelity simulations, and the result should still be helpful for high-fidelity simulations.
However, this extra cost is negligible for important applications such as neural networks.

\section{Linear response numerical examples: 40-dimensional Lorenz 96 system }
\label{s:numeric}

We use \Cref{t:adjInf} to compute the linear response of the stationary measure of the Lorenz 96 model \cite{Lorenz96} with multiplicative noise.
The dimension of the system is $M=40$.
The SDE is
\begin{eqnarray*}
  d x^i
  = \left( \left(x^{i+1}-x^{i-2}\right) x^{i-1} - x^i + \gamma^0 - 0.01 (x^i)^2 \right) dt + (\gamma^1 + \sigma(x)) dB^i
  \quad \textnormal{where}
  \\
  \sigma(x) = \exp{- |x|^2/2};
  \quad
  i=1, \ldots, M;
  \quad
  x_0 = [1,\ldots,1].
\end{eqnarray*}
Here $i$ labels different directions in $\R^M $, and it is assumed that $x^{-1}=x^{M-1}, x^0=x^M$ and $x^{M+1}=x^1$. 
We added noise and the $ - 0.01 (x^i)^2$ term, which prevents the noise from carrying us to infinitely far away.
Here, the parameter $\gamma^0$ controls the drift term and $\gamma^1$ controls the diffusion.
We consider the parameter region
\begin{equation*}\begin{split}
  \gamma^0\in [6,10],
  \quad \textnormal{} \quad 
  \gamma^1\in [2,6].
\end{split}\end{equation*}
The observable is 
\begin{equation*}\begin{split}
  \Phi(x) = |x|^2/M.
\end{split}\end{equation*}
The terms in \Cref{t:adjInf} become
\begin{equation*} \begin{split}
  \nabla \sigma(x) = - \sigma x, 
  \quad \quad
  \delta^0 F = [1,\ldots,1],
  \quad \quad
  \delta^1 \sigma = 1,
\end{split} \end{equation*}
where $\delta^i$ means taking derivative with respect to $\gamma^i$.
A typical orbit is in \Cref{f:orbit}.

\begin{figure}[ht] \centering
  \includegraphics[width=0.5\textwidth]{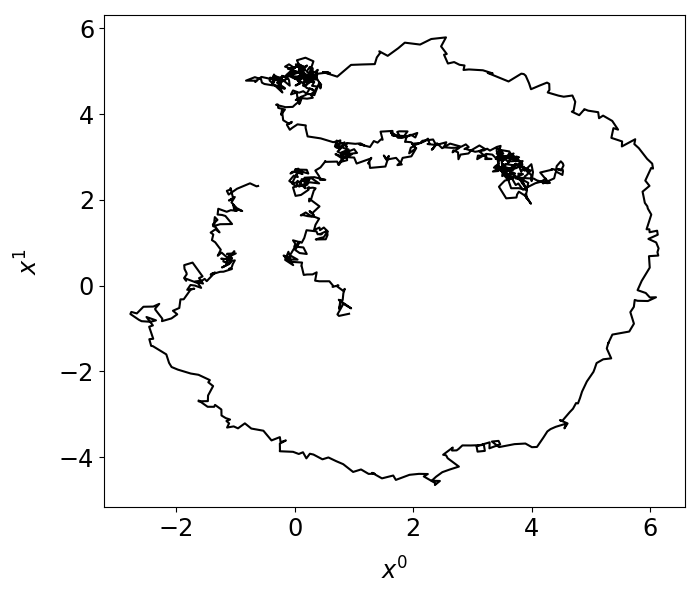}
    \caption{Plot of $x^0_t, x^1_t$ from a typical orbit of length $T=2$ at $\gamma^0=8, \gamma^1=2$. 
    }
  \label{f:orbit}
\end{figure}

Our goal is to compute the linear responses of the stationary measure with respect to the two parameters, and to see if it can be helpful for gradient-based optimization.
In our algorithm, we use the Euler integration scheme with $\Dt=0.002$, and set
\[
  \alpha_t \equiv 5
\]
to temper the unstableness.
In \Cref{t:adjInf}, we set $T=2000$ and $W=2$.

The derivatives with respect to each parameter are shown in \Cref{f:lorenz_line}.
As we can see, the algorithm gives accurate linear responses.
In particular, we plot $\Phi^{avg}$ computed on the original Lorenz system without noise.
The deterministic system seems to have no linear response: No one could prove it or compute it accurately.
However, if we add noise and compute the linear response of the noised system, the gradient is still very useful for the optimization of the original system.

\begin{figure}[ht] \centering
  \includegraphics[width=0.49\textwidth]{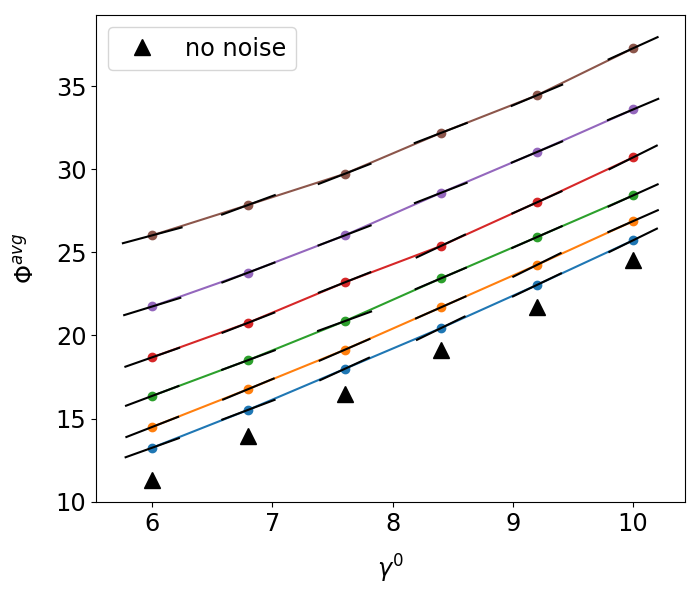}
  \includegraphics[width=0.49\textwidth]{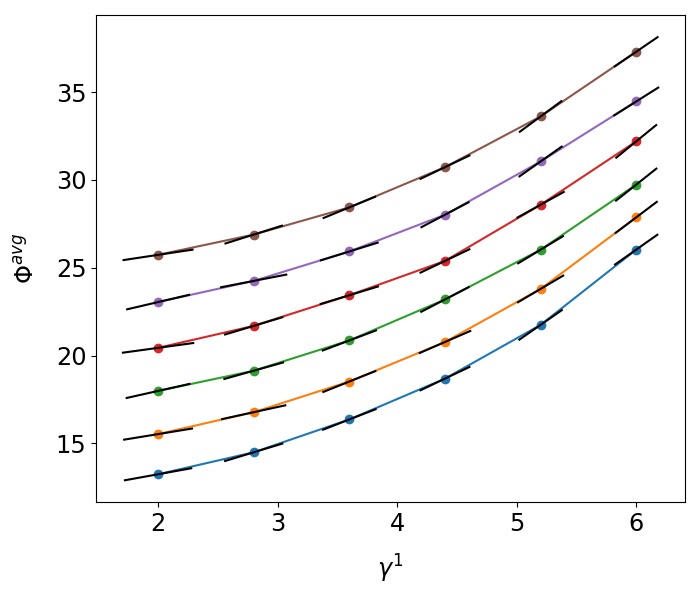}
  \caption{$\Phi^{avg}$ and $\delta \Phi^{avg}$ of the stationary measure.
  The dots are $\Phi^{avg}$, and the short lines are $\delta \Phi^{avg}$ computed by the adjoint path-kernel algorithm; they are computed from the same orbit of $T=1000$, $W=2$.
  Left: $\Phi^{avg}$ vs. $\gamma^0$, where each line is computed with a different $\gamma^1$.
  The black triangles are computed on the original Lorenz system without noise.
  Right: $\Phi^{avg}$ vs. $\gamma^1$, each line has a different $\gamma^0$.
  }
  \label{f:lorenz_line}
\end{figure}

Gradient vectors with respect to both parameters are shown in \Cref{f:lorenz_contour}.
As we can see, the gradient computed points to the ascent direction.
This enables gradient-based optimization.
Note that here each gradient consists of two derivatives, but we only need to run the adjoint algorithm only once to get the main term $\nu$, which is shared by the two parameters.
Hence, our adjoint path-kernel is suitable for cases with many parameters.

\begin{figure}[ht] \centering
  \includegraphics[width=0.5\textwidth]{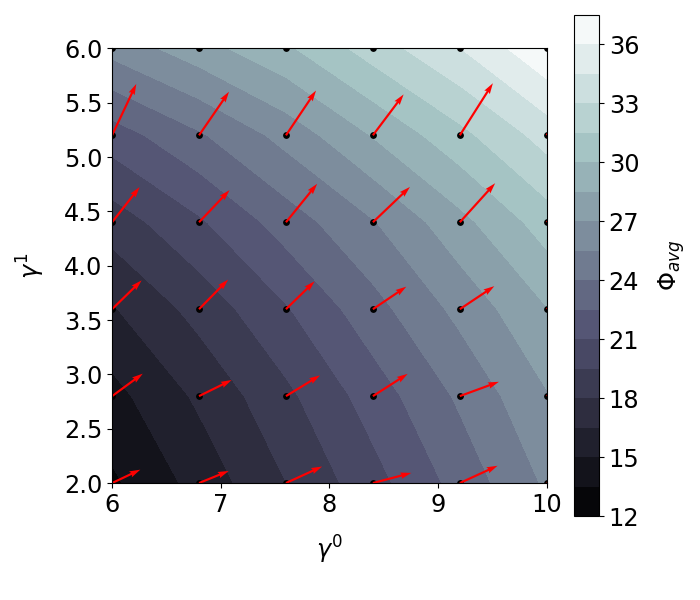}
  \caption{Gradients and the contour of $\rho(\Phi)$. The arrow is $1/10$ of the gradient.}
  \label{f:lorenz_contour}
\end{figure}

\section{Application in 4D-Var data assimilation}
\label{s:assi}

The adjoint path-kernel method enables a solution to the difficult version of the 4D variational data assimilation, where a sample \textit{low-dimensional} observation path of an SDE is given, and we try to infer the unknown parameters and the original path.
This problem is difficult for all the other known methods due to the lack of gradient tools. 
It is also difficult for the authors' other methods such as the tangent path-kernel method in \cite{dud} and the divergence-kernel method in \cite{DKlinR}.

\subsection{Problem description}
\label{s:assi_desp}

Given a path of observations $\{y_t\}_{t\in[0,T]}$, we assume that it comes from the system
\[
dx^\gamma_t
= F^\gamma(x^\gamma_t) dt 
+ \sigma^\gamma(x^\gamma_t) dB,
\quad \textnormal{} \quad 
y_t = \phi(x_t),
\]
where the parameter $\gamma$ and the initial condition $x_0$ are to be inferred.
Here $x\in \R^M, y\in \R^{N}$, where typically $N < M$.
This problem is called the data assimilation problem.
Previously, solving this problem typically required decomposing the time span into small segments and performing inference/update/optimization on each segment.
Another workaround is to use a very informative observation $\phi$, which gives sufficient data within a short time to make a good guess on the state $x_t$, thus reducing the difficulty.

With our new tool, we can solve the 4D-Var optimization problem on a long-time window for chaotic systems from low-dimensional observations. 
More specifically, we solve
\begin{equation}\begin{split}
\label{e:4dvar}
  \min_{\gamma, x_0, \xi_t} \E{ \Phi(x_{[0,T]}) }
  \quad \textnormal{where} \quad 
  \Phi(x_{[0,T]})
  :=\frac 1{2T} \int_0^T |\phi(x_t)-y_t|^2 + C |\xi|^2 dt
  \\
  \textnormal{such that} \quad 
  dx_t
  = F^\gamma (x_t) dt 
  + \sigma^\gamma dB 
  + \xi_t dt.
\end{split}\end{equation}
Here $x_0$ is deterministic and needs to be optimized.
$\xi_t$ can be any adapted process describing the error in the SDE; it can be a function of $x_{0\sim t}$.
For simplicity, here we assume the form 
\begin{equation*}\begin{split}
  \xi_t := \xi^\gamma(x_t, x_t^*) = \gamma' |x_t^*-x_t|^2 (x_t^*-x_t),
\end{split}\end{equation*}
where $x_t^*\in\R^M$ is deterministic and needs to be optimized with the parameter $\gamma'$.
We also assume for simplicity that $\sigma^\gamma$ does not depend on $x$ but its value needs to be optimized, so we denote it just by $\sigma$.
All these simplified model can be replaced by more complicated models such as neural networks; we focus on a conceptual demonstration and leave those more complicated models to future papers.

We explain how to get the constant coefficient $C$ in the loss function $\E{\Phi}$.
This coefficient is meant to balance the two parts of the loss function: the discrepancy in the observation $\phi-y$ and the outside force $\xi$.
But the two parts have different units.
In particular, the size of $\phi-y$ should be compared with the observation $\phi$, and $\xi$ should be compared with $F$.
Hence, in practice, we can set
\begin{equation*}\begin{split}
  C \approx  \int_0^T |\phi|^2 dt \left(\int_0^T |F|^2 dt \right)^{-1},
\end{split}\end{equation*}
so that the two parts in the loss function are properly scaled and are comparable.

\subsection{Adapting adjoint path-kernel method to this case}
\label{s:assi_adj}

First, we need to modify our formulas a bit so that the loss is the form of a time integration.
For convenience we directly work with time-discretized SDEs.
We shall also overload our notation and use $\delta$ to denote a first-order small perturbation: this is perhaps clearer, since we now want to directly optimize several quantities such as $x_0$ and $x_n^*$, which are parameters themselves (in particular, they are not parameterized by $\gamma$'s).

The tangent path-kernel formula in \Cref{t:tan_discrete} becomes: for 
\begin{equation} \label{e:wan}
  x_{n+1}
  =x_n + F^\gamma (x_n) \Dt + \sigma \DB_n + \xi^\gamma(x_n, x_n^*) \Dt,
  \quad
  \DB_n \overset{i.i.d.}{\sim} \cN(0,\Dt). 
\end{equation}
We use $\nabla_\gamma$ to denote the gradient in $\gamma$, $\nabla_*$ to denote the gradient in $x^*$. 
In our simplified model, $\delta \xi = \nabla \xi \delta x + \nabla_\gamma\xi \delta \gamma + \nabla_*\xi \delta x^*$.
Let $\delta x_n$ be the solution of the following tangent equation starting from $\delta x_0$;
\begin{equation*}\begin{split}
  \delta x_{n+1} 
  = (1 - \alpha_n \Dt) \delta x_n
  + \nabla F_n  \delta x_n \Dt
  + \nabla_\gamma F_n \delta \gamma \Dt
  + \delta \sigma \DB_n
  + \delta \xi_n \Dt
  \\
  =M_n \delta x_n + p_{n+1},
  \quad \textnormal{where} \quad 
  M_n:= (1 - \alpha_n \Dt) I + \nabla F_n\Dt + \nabla \xi_n \Dt,
  \\
  p_{n+1} := \nabla_\gamma F_n \delta \gamma \Dt
  + \delta \sigma \DB_n
  + (\nabla_\gamma\xi_n \delta \gamma' + \nabla_*\xi_n \delta x_n^*) \Dt.
\end{split}\end{equation*}
Denote $\tilde \Phi := \Phi(x_0,\ldots, x_N) - \E{\Phi}$, the linear response has the expression
\begin{equation*}\begin{split}
\delta \E{\Phi}
= 
\frac 1T \E{  
\sum_{n=0}^{N-1}
\left(
(\phi_n - y_n)^T \nabla\phi_n \Dt 
+ \frac{1}{\sigma} \tilde \Phi  \alpha_n \DB_n^T 
\right) \delta x_n
+ C \xi_n \cdot \delta \xi_n \Dt }
\\
= \frac 1T \E{  
\sum_{n=0}^{N-1}
\omega_n \cdot \delta x_n
+ C \xi_n \cdot (\nabla_\gamma\xi_n \delta \gamma' + \nabla_*\xi_n \delta x_n^* ) \Dt
}
, \quad \textnormal{where} \\ 
\omega_n := \nabla\phi_n^T (\phi_n - y_n) \Dt 
+ \frac{1}{\sigma} \tilde \Phi  \alpha_n \DB_n 
+ C \nabla \xi_n^T \,\xi_n \Dt
.
\end{split}\end{equation*}
Then we can obtain the adjoint formulation by the same procedure as in \Cref{s:derive}.
The linear response is expressed by
\begin{equation*}\begin{split}
\delta \E{\Phi }
= \frac 1T \E{\nu_0 \cdot \delta x_0
+ \sum_{k=0}^{N-1} \nu_{k+1} \cdot p_{k+1}
+ C \xi_k \cdot (\nabla_\gamma\xi_k \delta \gamma' + \nabla_*\xi_k \delta x_k^*) \Dt }
\end{split}\end{equation*}
where the adjoint process is defined below with the expanded linear response formula.

\begin{formula} [gradients for 4D-Var] \label{t:assimi}
The linear response has the expanded formula
\begin{equation*}\begin{split}
  \delta \E{\Phi }
  = \frac 1T \E{\nu_0 \cdot \delta x_0
  + \delta \gamma \cdot \sum_{k=0}^{N-1} \nabla_\gamma F_k^T \nu_{k+1} \Dt
  + \delta \sigma \sum_{k=0}^{N-1} \DB_k ^T \nu_{k+1} \right.
  \\ \left.
  + \delta \gamma' \sum_{k=0}^{N-1} \nabla_\gamma \xi_k ^T (\nu_{k+1} + C \xi_k) \Dt
  + \sum_{k=0}^{N-1} \delta x_k^* \cdot \nabla_*\xi_k^T \, (\nu_{k+1} + C \xi_k) \Dt}
\end{split}\end{equation*}
where the backward process is 
\begin{equation*}\begin{split}
  \quad \textnormal{} \quad 
  \nu_{k}
  = M_k^T\nu_{k+1} + \omega_k
  = (1 - \alpha_k \Dt)\nu_{k+1} 
  + \nabla F_k^T \nu_{k+1} \Dt 
  + \nabla \xi_k^T \nu_{k+1} \Dt
  \\
  + \nabla\phi_k ^T (\phi_k - y_k) \Dt
  + \frac{ 1}{\sigma} \tilde \Phi \alpha_k \DB_k
  + C \nabla \xi_k^T \,\xi_k \Dt
\end{split}\end{equation*}
with terminal condition $ \nu_{N} = 0 $.
\end{formula}

\subsection{Optimization via stochastic gradient descent}
\label{s:assi_opti}

With the new gradient tool, we can use a very simple strategy to solve the 4D-Var problem in \Cref{e:4dvar}: stochastic gradient descent.
We explain this and how to select the step size, or learning rate, $\eta$ for each parameter.

In each iteration, we update all parameters according to
\begin{equation*}\begin{split}
\delta x_0 = -\frac {\eta_1} T \E{\nu_0} ,\quad
\delta \gamma = -\frac {\eta_2} T \E{ \sum_{k=0}^{N-1} \nabla_\gamma F_k^T \nu_{k+1} \Dt}, \quad
\delta \sigma = -\frac {\eta_3} T \E{\sum_{k=0}^{N-1} \DB_k ^T \nu_{k+1}},
\\
\delta \gamma'= -\frac {\eta_4} T \E{\sum_{k=0}^{N-1} \nabla_\gamma \xi_k^T (\nu_{k+1} + C \xi_k) \Dt}, \quad
\delta x_k^* = -\frac {\eta_5} T \E{ \nabla_*\xi_k^T (\nu_{k+1} + C \xi_k)} 
\end{split}\end{equation*}
where $\eta >0$ is the stepsize, or learning rate.
Up to the order of $O(\eta)$, this guarantees a reduction in the overall loss functions for small $\eta$.
To compute the expectation in the above formulas, we first compute $L$ many sample paths, then compute the backpropagation in \Cref{t:assimi} on each path, then average over all sample paths.
For the later numerical example, we can pick a small $L$ such as 10.

We can set $\eta_i=1$ in general and then modify it strategically.
First, we do not want $\eta_i$'s to be too large so that we are out of the linear region where the parameter-gradient is meaningful -- this is an idea explored in \cite{linearRange_GD}.
The criterion is to limit, say $\eta_0$ for example, so that 
\begin{equation*}\begin{split}
  \pp{L}{x_0}  \delta x_0 = \eta_0 \left(\pp{L}{x_0}\right)^2 \le 0.1 \E{\Phi}
\end{split}\end{equation*}
This means that the projected reduction in the loss due to updating $x_0$ is at most 10\%.

Another upper bound for $\eta_3$ and $\eta_4$ comes from the fact that $\sigma$ and $\gamma'$ must be positive by definitions.
Also, $\sigma$ should not be too small; otherwise our path-kernel estimator becomes large and thus requiring more samples to take the average.
Hence, we require
\begin{equation*}\begin{split}
  \sigma + \delta \sigma \ge \sigma_{min}, \qquad
  \gamma' + \delta \gamma' \ge \gamma'_{min}.
\end{split}\end{equation*}
In the later example we set $\sigma_{min}=0.5, \gamma'_{min}=0.1$.

The cubic model for $\xi$ can cause numerical instabilities when integrating the primal ODE in \Cref{e:wan} by Euler forward scheme.
This is especially problematic at the beginning phase of the optimization where $x_t$ and $x^*_t$ are far away from each other, and the cubic model in $\xi$ can get very large.
There are two solutions to avoid this issue.
First, we can impose a separate restriction on $\gamma'$ or $\eta_4$ so that $\xi$ does not get too large.
Second, we can set the initial guess of $x_0$ and $x_n^*$'s close to each other.
In the later example, we use the second solution, and $x_0$ and all $x_n^*$'s are initialized at the origin.

\subsection{Numerical example: 3D Lorenz with 2D observations}
\label{s:assi_3D}

We demonstrate our method on a 3D Lorenz 63 system whose equation is
\begin{equation*}\begin{split}
  d \begin{bmatrix} x^0 \\ x^1 \\ x^2 \end{bmatrix}
  = \begin{bmatrix} 10(x^1-x^0) \\ 
  x^0(\gamma-x^2)-x^1 \\ 
  x^0 x^1 - \frac 83 x^2 \end{bmatrix} dt
  + \sigma dB + \gamma'|x_t^*-x_t|^2 (x_t^*-x_t) dt, \qquad
  \phi (x) = \begin{bmatrix} x^1 \\ x^2 \end{bmatrix}.
\end{split}\end{equation*}
Note that here the superscript for directions starts from $0$.
We use the forward-Euler scheme to integrate this SDE, with $\Dt=0.002$.

The path of observation, or data, is obtained by solving the original Lorenz system without noise and outside force $\xi$, which is obtained by setting $\sigma=0$ and $\gamma'=0$.
Then we set 
\begin{equation*}\begin{split}
\gamma_{true}=28, \qquad
x_{true} = [-10, -15, 20]
\end{split}\end{equation*}
and simulate the deterministic path for time $T$ and record the data $\{y_n \}_{ n=0}^N$.
This is the data we try to match via gradient descent.
After we obtain the data, we forget the true parameters.

We explain how to set the hyper-parameters.
For the constant in the loss function, we set $C=1/150$ according to the discussion in \Cref{s:assi_desp}.
For the optimization, we start from initial guess
\begin{equation*}\begin{split}
\gamma = 33, \quad
x_0=[0,0,0], \quad
x^*_n=[0,0,0] \textnormal{ for all } n, \quad
\sigma = 2, \quad
\gamma' = 0.1.
\end{split}\end{equation*}
Our initial guesses are largely blind, because there is no meaningful guess on the missing dimension by looking a short-time observation; this is even more so since $\gamma$ is also unknown.
In the adjoint path-kernel method, we set $\alpha \equiv 5$, which is larger than the top Lyapunov exponent (for example, at $\gamma=1$, the top Lyapunov exponent is about $1$).
The number of samples used for each update is $L=10$.
The stepsize for the gradient descent is set according to the discussion in \Cref{s:assi_opti}.

We first show results in the case $T=2$; the path is shorter so the problem is easier, but the plots are more clear.
The 2D observation paths at different stages are shown in \Cref{f:assi path T2}.
As we can see, the observation path generated by the inferred SDE gradually matches the data during the optimization.
This shows that our method is correct.
The wall time for each parameter-update on a 3GHz CPU is less than 0.3 second.
We also plot the history of $\Phi^{avg}$ for $L=10$ and $L=100$ in \Cref{f:assi hist T2}.
As we can see, a larger $L$ improves the convergence speed and reduces the fluctuation in the optimization, but the benefit may not worth the cost: this is similar to the general case of stochastic gradient descent.

\begin{figure}[ht] \centering
  \includegraphics[width=0.32\textwidth]{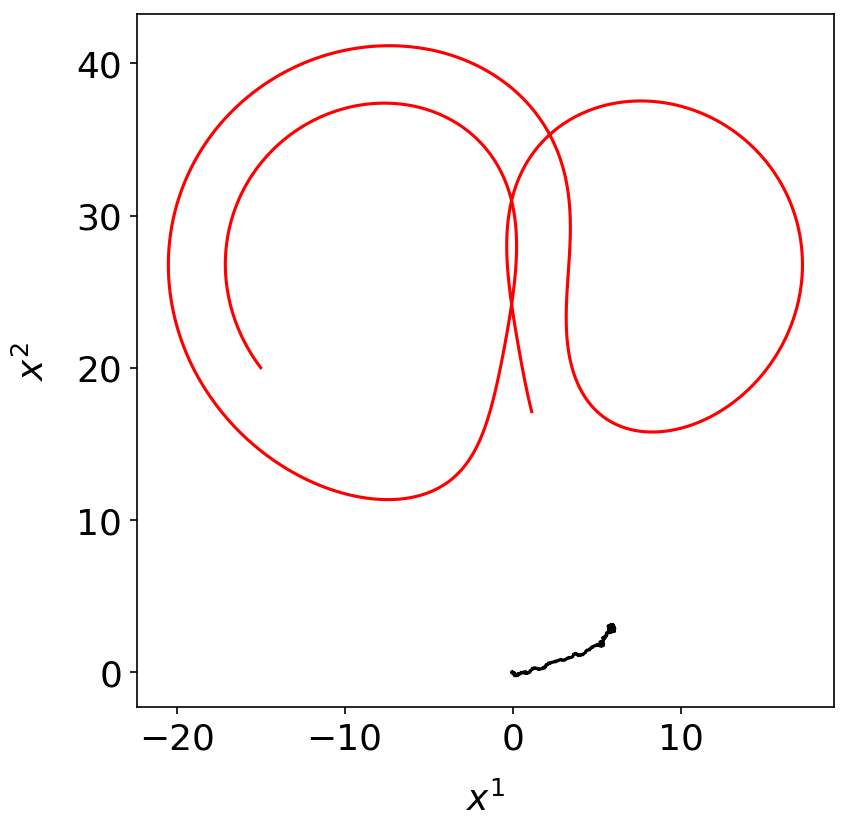}
  \includegraphics[width=0.32\textwidth]{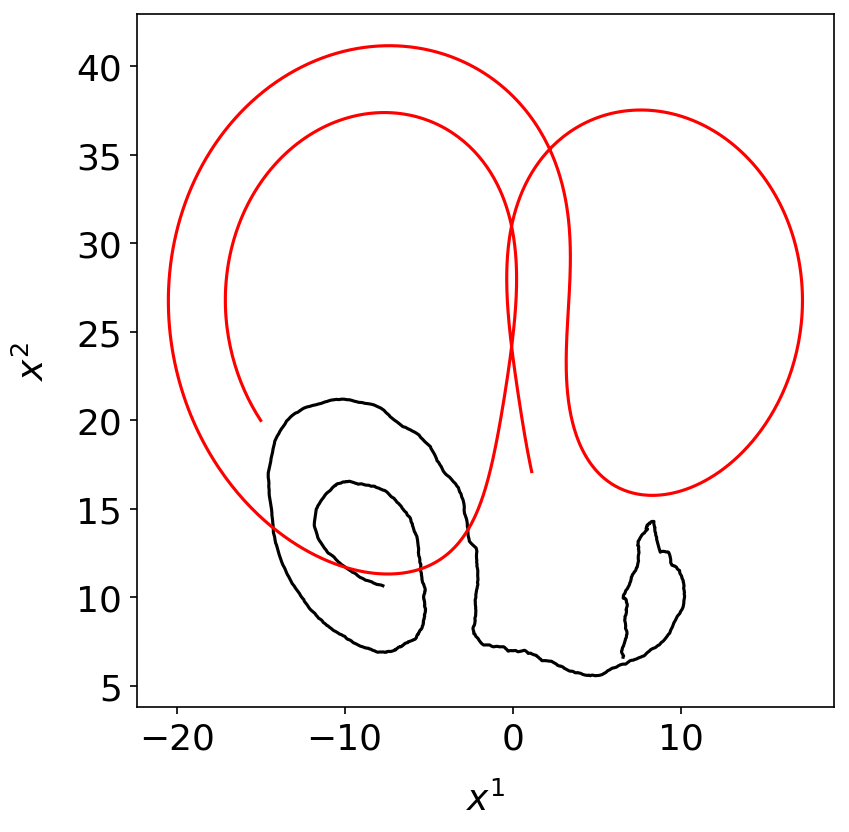}
  \includegraphics[width=0.32\textwidth]{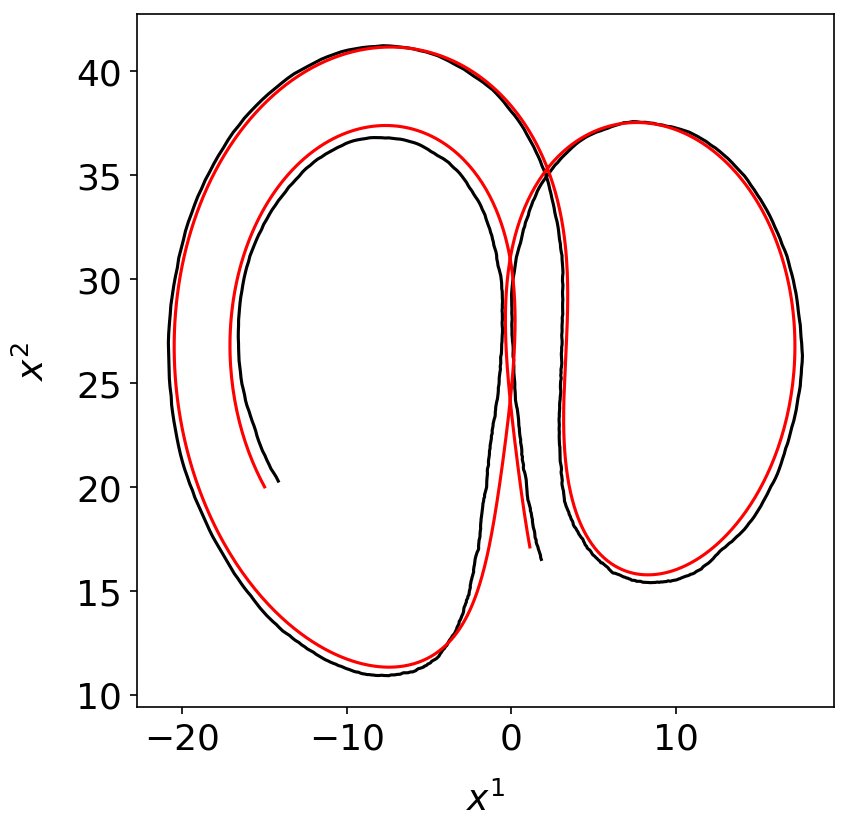}
  \caption{Lorenz 63 example. Comparison of the 2D data (red line) and the observation generated by the 3D model (black line, averaged over $L=10$ samples) at different stages of the optimization.
  From left to right: initial guess, after 21 updates (or `epoch'), and after 291 updates.}
  \label{f:assi path T2}
\end{figure}

\begin{figure}[ht] \centering
  \includegraphics[width=0.48\textwidth]{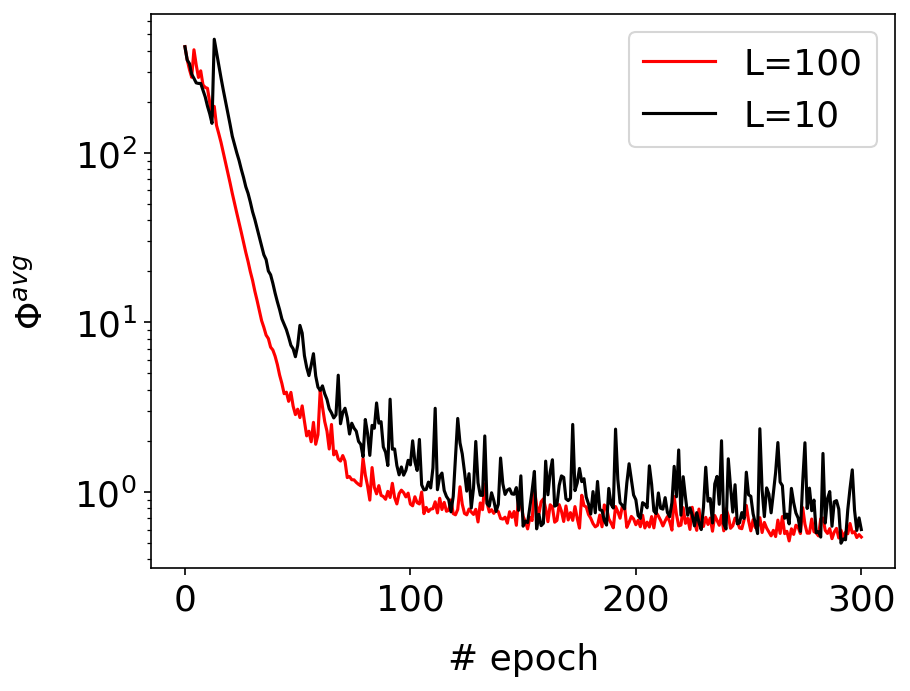}
  \caption{
  $\Phi^{avg}$ vs. number of updates for $T=2$ and different sample size $L$.
  }
  \label{f:assi hist T2}
\end{figure}

Then we show results in the case $T=20$ and $L=10$; the path is longer so the problem is more difficult (some would say that the problem is more `stiff') and the conventional adjoint method becomes unaffordable.
We update the parameter 1000 times by gradient descent.
The 2D observation paths after the optimization is shown in \Cref{f:assi path T20}.
We also plot the history of $\Phi^{avg}$.
This shows that our result can work over a long-time effectively and efficiently.
In contrast, conventional adjoint method is expected to blow up because Lyapunov exponent $\lambda\approx 1$ and $T = 20$, making conventional gradients grow like $e^{\lambda T}\approx 5\times 10^8$; this is exactly what we circumvent by damping with $\alpha$ and the adjoint path-kernel estimator.

\begin{figure}[ht] \centering
  \includegraphics[width=0.50\textwidth]{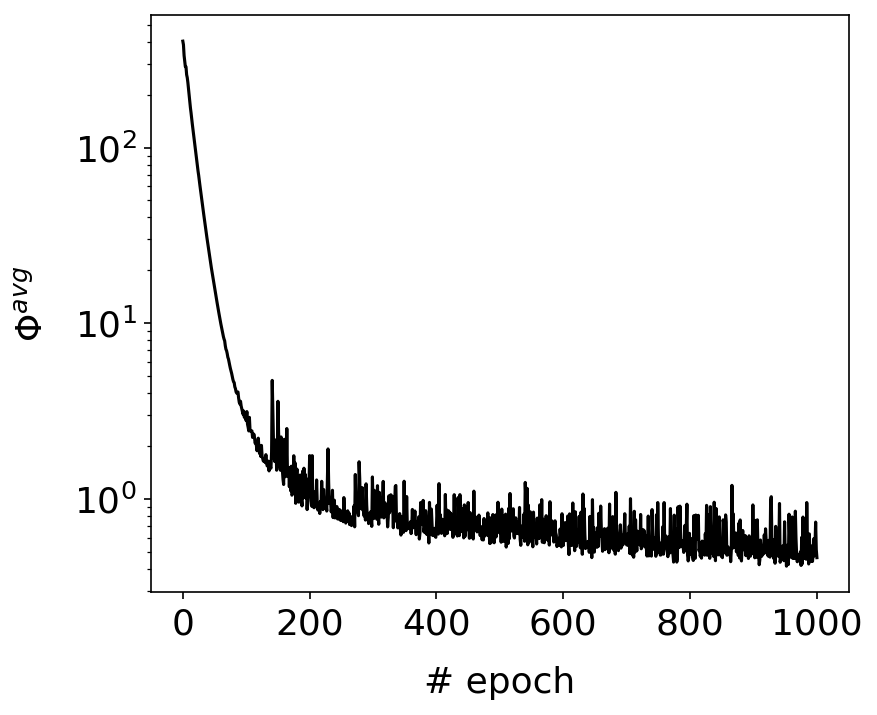}
  \includegraphics[width=0.415\textwidth]{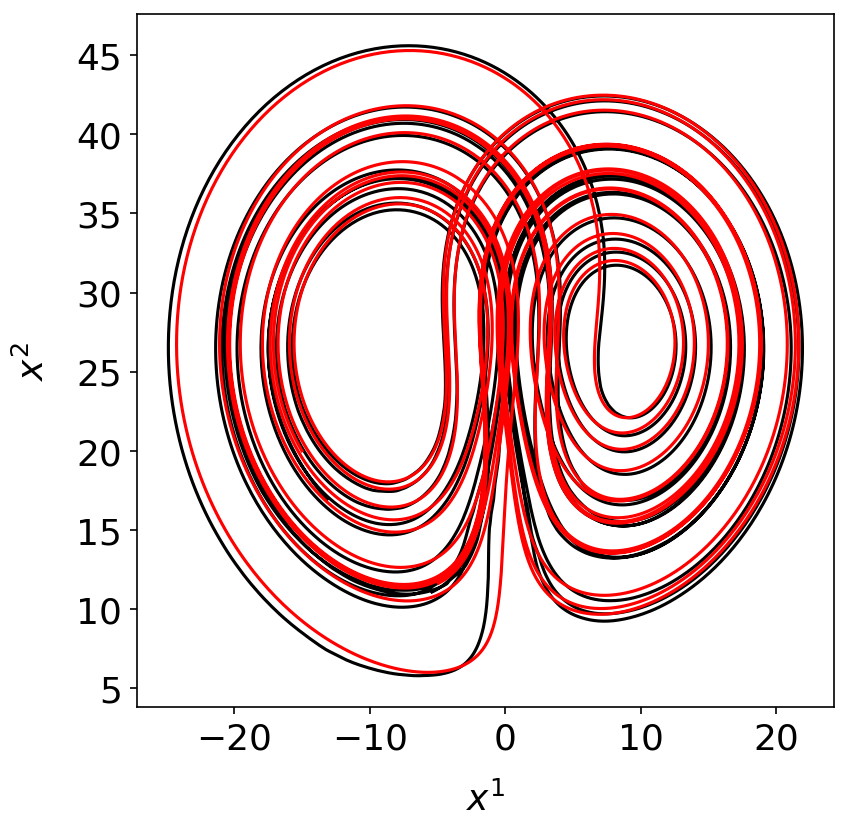}
  \caption{
  Left: $\Phi^{avg}$ history of Lorenz 63, $T=20$.
  Right: data and the sample generated by the deterministic part of the optimized model.
  }
  \label{f:assi path T20}
\end{figure}

At the end of the optimization, $ \sigma = 0.5$ takes the lower bound value.
This verifies that our loss function penalizes noise and guide the optimization towards a deterministic solution.
So we can set $\sigma=0$ and compute only the deterministic part of the inferred dynamics.
This further removes the noise contribution in the loss, and reduces the loss to 
\begin{equation*}\begin{split}
  \Phi = 0.24.
\end{split}\end{equation*}
The inferred deterministic path is in \Cref{f:assi path T20}: this shows that the correction term is large enough to control the chaos.
In the future, we can consider switching to deterministic dynamics and perform further refinement of the optimization.
However, the most difficult step is to first arrive at a good guess, which must be achieved by adjoint path-kernel method.

Finally, we remark that this seemingly simple problem is difficult for other methods.
In the context of data assimilation, previous methods typically avoid optimizing \textit{one} loss function with low-dimensional observation, since the system is unstable and we need to assimilate a long-time observation to infer information about the missing dimensions.
In the context of gradient-based methods, where the gradients are computed by linear response methods, this is also difficult.
The pure kernel-differentiation method does not work because we want to optimize the initial distribution of the path which is singular, and because we want to optimize the noise.
The pure divergence method does not work because of singular initial conditions; moreover, some divergence methods involve approximation on a set of basis, which does not work even in moderately high-dimensional cases such as our next example.
Our divergence-kernel method in \cite{DKlinR} does not work well due to optimizing singular initial conditions, although high-dimension and optimizing noise do not cause problems.
The tangent path-kernel method in \cite{dud} does not work well, since there are too many parameters, most of them in $x^*_n$'s, to optimize.

\subsection{Numerical example: 10D Lorenz 96 system with 8D observations}
\label{s:assi_10D}

We demonstrate our method on a moderately high-dimensional example, a 10D Lorenz 96 system whose equation is
\begin{equation*}\begin{split}
 d x^i
  = \left(x^{i+1}-x^{i-2}\right) x^{i-1} - x^i + \gamma
  + \sigma dB^i + \gamma'|x_t^*-x_t|^2 (x_t^{*i}-x_t^i) dt,
  \quad \textnormal{where}
  \\
  \quad i=0, \ldots, 9;
\end{split}\end{equation*}
We use the forward-Euler scheme to integrate this SDE, with $\Dt=0.002$.
The observation misses two coordinates, that is,
\begin{equation*}\begin{split}
  \phi(x) = [x^0\sim x^3, x^5\sim x^9 ].
\end{split}\end{equation*}
The path of observation is obtained by solving the original Lorenz system with $\sigma=0$ and $\gamma'=0$, and
\begin{equation*}\begin{split}
\gamma_{true}=8, \qquad
x_{true} = [-6.9,-0.5,1.5,9.3,0.9,1.3,0.2,2.6,6.7,2.7].
\end{split}\end{equation*}
We consider the total time $T=5$.

For the hyper-parameters, we set $C=1/150$.
The initial guess is
\begin{equation*}\begin{split}
\gamma = 13, \quad
\sigma = 2, \quad
\gamma' = 0.1, \quad
x_0=0, \quad
x^*_n=0,
\end{split}\end{equation*}
where $0$ is the zero vector.
In the adjoint path-kernel method, we set $\alpha \equiv 3$, which is larger than the top Lyapunov exponent.
The number of samples used for each update is $L=10$.
The stepsize for the gradient descent is set to $\eta=1$ at the start and then $\eta=0.1$ once $\E{\Phi}\le 6$, and also modified according to the discussion in \Cref{s:assi_opti}.
The wall time for each parameter-update on a 3GHz CPU is less than 1 second.

We plot $[x^1, x^2]$ of the observation paths at different stages in \Cref{f:10d assi path}.
As we can see, the observation path generated by the inferred SDE gradually matches the data during the optimization.
Also note that the two coordinates do not have an obvious relation, which, again, implies that we cannot reasonably infer the missing coordinate from adjacent coordinates.
We also plot the history of $\Phi^{avg}$ in \Cref{f:10d assi hist}.
At the end of the optimization, the parameters in the dynamic and the initial condition are
\begin{equation*}\begin{split}
\gamma = 8.0, \quad
x_0=[-6.9,-0.5,1.4,9.3,0.9,1.3,0.2,2.6,6.8,2.7],
\end{split}\end{equation*}
which are basically the same as the true values.
This shows that our result can work over a long-time effectively and efficiently even in high dimensions.

\begin{figure}[ht] \centering
  \includegraphics[width=0.32\textwidth]{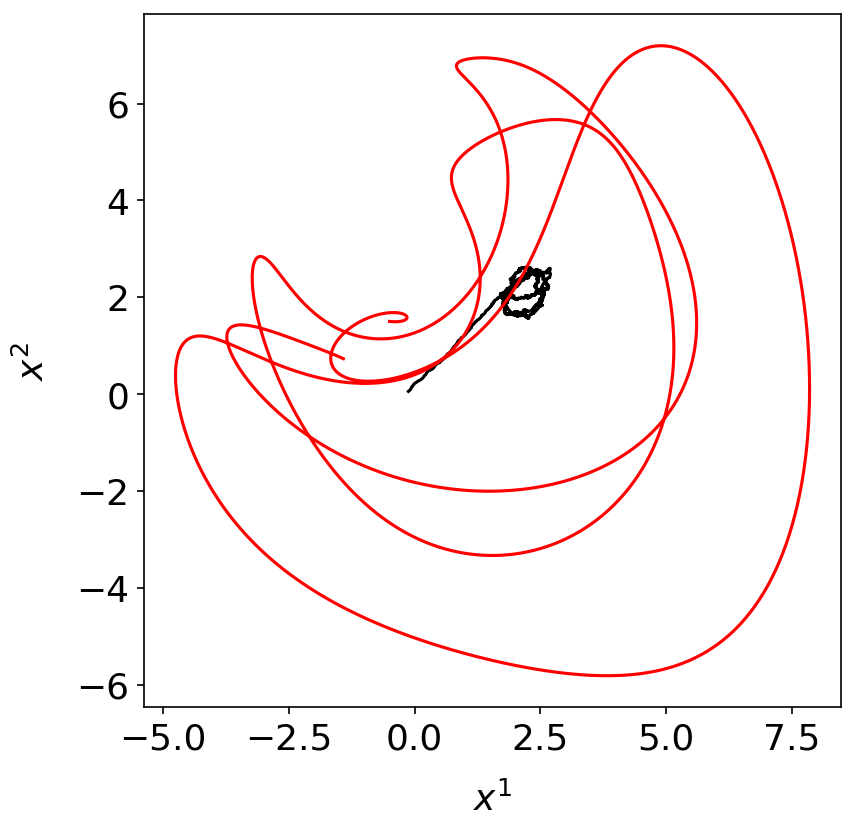}
  \includegraphics[width=0.32\textwidth]{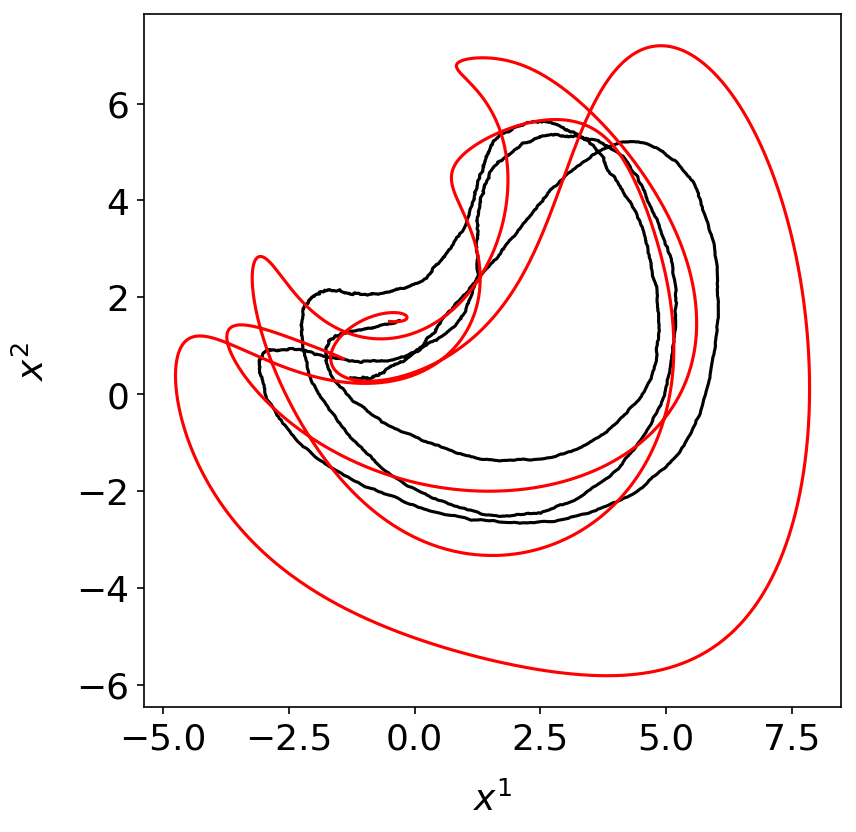}
  \includegraphics[width=0.32\textwidth]{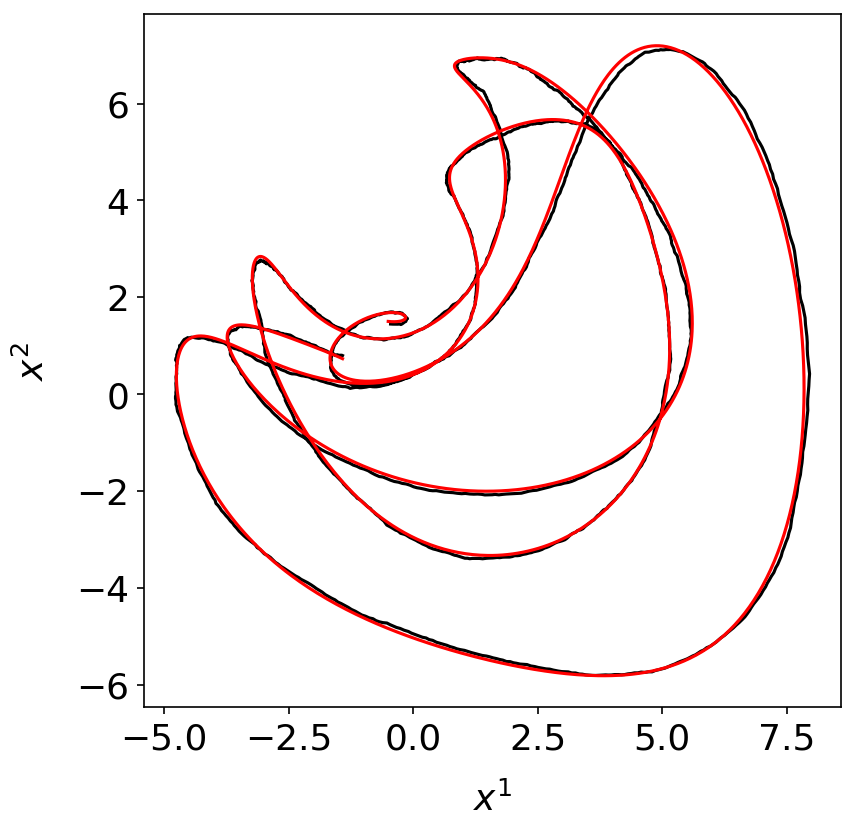}
  \caption{Lorenz 96 example. Comparison of the $[x^1, x^2]$ coordinates of the data (red line) and the generated sample (black line, averaged over $L=10$ sample paths with noise) at different stages of the optimization. 
  From left to right: after 1 update, after 15 updates, after 1922 updates.}
  \label{f:10d assi path}
\end{figure}

\begin{figure}[ht] \centering
  \includegraphics[width=0.50\textwidth]{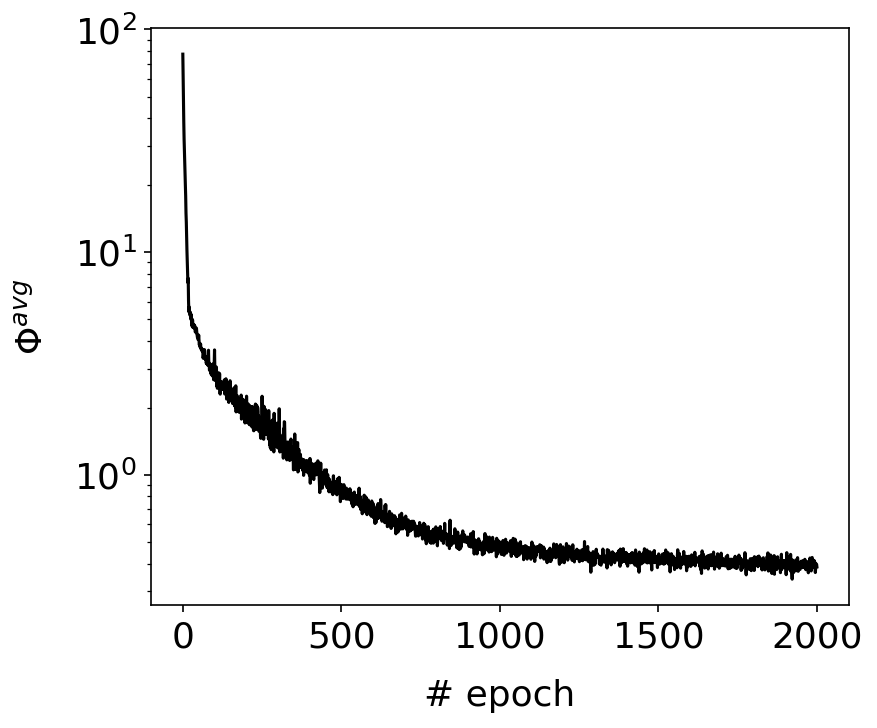}
  \includegraphics[width=0.415\textwidth]{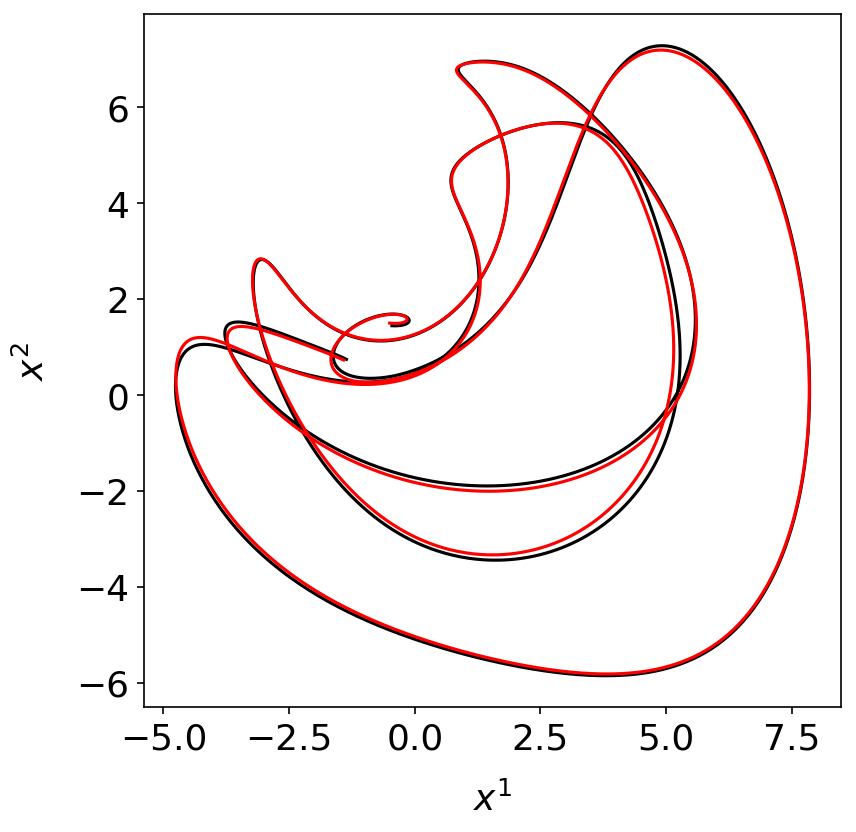}
  \caption{
  Left: $\Phi^{avg}$ vs. number of updates of the 10D assimilation problem.
  Right: data and the sample generated by the deterministic part of the optimized model.
  }
  \label{f:10d assi hist}
\end{figure}

At the end of the optimization, $ \sigma = 0.5$ takes the lower bound value.
Moreover, $\gamma' = 3.8$ is large enough to control the chaos by the correction term, so we can set $\sigma=0$, and only compute the deterministic part of the inferred dynamics.
This reduces the loss to 
\begin{equation*}\begin{split}
  \Phi = 0.050.
\end{split}\end{equation*}
The inferred deterministic path is in \Cref{f:10d assi hist}.

\section*{Acknowledgements}

The author thanks Zhuoyuan Li, Zeju Sun, Pingwen Zhang for helpful discussions.

\section*{Declarations}

\oldsubsection*{Data availability statement}

The code used in this paper is posted at \url{https://github.com/niangxiu/APK}.
There are no other associated data.

\oldsubsection*{Funding and Conflicts of interests/Competing interests}

The author has no relevant conflicting or competing interests.

%\begin{appendix}

%\end{appendix}

\bibliographystyle{abbrv}
{\footnotesize\bibliography{library}}

\end{document}